\documentclass[a4paper,11pt]{amsart}

\usepackage{amsmath}
\usepackage{eucal}
\usepackage{amssymb}
\usepackage{xypic}

\newtheorem{theorem}{Theorem}[section]
\newtheorem{proposition}{Proposition}[section]
\newtheorem{corollary}{Corollary}[section]
\newtheorem{lemma}{Lemma}[section]
\theoremstyle{definition}
\newtheorem{definition}{Definition}[section]
\theoremstyle{remark}
\newtheorem{remark}{Remark}

\DeclareMathOperator{\rank}{\text{\rm rank}}
\renewcommand{\tilde}{\widetilde}

\begin{document}

\title[Invariant jet differentials on projective hypersurfaces]{Existence of global invariant jet differentials on projective hypersurfaces of high degree}

\author{Simone Diverio}
\address{Istituto \lq\lq Guido Castelnuovo\rq\rq{},
SAPIENZA Universit\`a di Roma}
\email{diverio@mat.uniroma1.it} 
\address{Institut \lq\lq Fourier\rq\rq{},
Universit\'e de Grenoble I}
\email{sdiverio@fourier.ujf-grenoble.fr}

\maketitle

\begin{abstract}
Let $X\subset\mathbb P^{n+1}$ be a smooth complex projective hypersurface. In this paper we show that, if the degree of $X$ is large enough, then there exist global sections of the bundle of invariant jet differentials of order $n$ on $X$, vanishing on an ample divisor. We also prove a logarithmic version, effective in low dimension, for the log-pair $(\mathbb P^n,D)$, where $D$ is a smooth irreducible divisor of high degree. Moreover, these result are sharp, \emph{i.e.} one cannot have such jet differentials of order less than $n$.
\end{abstract}

\section{Introduction}

Let $X$ be a compact complex manifold. According to a well-know criterion of Brody, $X$ is Kobayashi hyperbolic if and only if there are no non-constant entire holomorphic curves in $X$. In 1970, S. Kobayashi \cite{Kobayashi70} conjectured that if $X\subset\mathbb P^{n+1}$ is a generic hypersurface of degree \hbox{$d=\deg X$} at least equal to $2n+1$, then $X$ is Kobayashi hyperbolic (analogously, he proposed also the following logarithmic version of his conjecture: if $D\subset\mathbb P^n$ is a generic irreducible divisor of degree $\deg D\ge 2n+1$, then $\mathbb P^n\setminus D$ is Kobayashi hyperbolic). Thus proving the \lq\lq compact\rq\rq{} Kobayashi conjecture is equivalent to proving the non existence of entire holomorphic curves on a generic projective hypersurface of degree large enough.

Several decades after the pioneering work of Bloch in 1926, it has been realized that an essential tool for controlling the geometry of entire curves on a manifold $X$  is to produce differential equations on $X$ that every entire curve must satisfy. For instance, in 1979, Green and Griffiths \cite{G-G79} constructed the sheaf $\mathcal J_{k,m}$ of jet differentials of order $k$ and weighted degree $m$ and were able to prove the Bloch conjecture ({\it i.e.}, that every entire holomorphic curve in a projective variety is algebraically degenerate as soon as the irregularity is greater than the dimension). 

Several years later, Siu outlined new ideas for proving Kobayashi's conjecture, by making use of jet differentials and by generalizing some techniques due to Clemens, Ein and Voisin (see \cite{Siu04}). However many details are missing, and it also seems to be hard to derive effective results from Siu's approach.

We would like here to concentrate on a refined and more geometrical version of the bundle of Green and Griffiths, namely the bundle of \emph{invariant jet differentials}, which was first introduced in this context by Demailly in \cite{Demailly95}. This bundle reflects better the geometry of entire curves since it just takes care of the image of  such curves and not of the way they are parametrized: following \cite{Demailly95}, we will denote it $E_{k,m}T^*_X$.
The general philosophy is that global holomorphic sections of $E_{k,m}T^*_X$ vanishing on a fixed ample divisor give rise to global algebraic differential equations that every entire holomorphic curve must satisfy.

It is known by \cite{D-EG00} that every smooth surface in $\mathbb P^3$ of degree greater or equal to $15$ has such differential equations of order two. For the dimension three case, Rousseau \cite{Rousseau06b} observed that one needs to look for order three equations since one has in general the vanishing of symmetric differentials and invariant $2$-jet differentials for smooth hypersurfaces in projective $4$-space. On the other hand \cite{Rousseau06b} shows the existence of global invariant $3$-jet differentials vanishing on an ample divisor on every smooth hypersurface $X$ in $\mathbb P^4$, provided that $\deg X\ge 97$. 

Recently, in \cite{Div08}, we improved the bound for the degree obtained in \cite{Rousseau06b} and found the existence of invariant jet differentials for smooth projective hypersurfaces of dimension at most $8$ (with an explicit effective lower bound for the degree of the hypersurface up to dimension $5$). 

Until our paper \cite{Div08}, the existence was obtained by showing first that the Euler characteristic $\chi(E_{k,m}T^*_X)$ of the bundle of invariant jet differentials is positive for $m$ large enough. Then, with a delicate study of the even cohomology groups of such bundles -- which usually involves the rather difficult investigation of the composition series of $E_{k,m}T^*_X$ -- one could obtain in principle a positive lower bound for $h^0(X,E_{k,m}T^*_X)$ in terms of the Euler characteristic.

Here we generalize the result of \cite{Div08} to arbitrary dimension, thus solving the problem of finding \emph{invariant} jet differentials on complex projective hypersurfaces of high degree. Namely we get the following.

\begin{theorem}\label{existence} 
Let $X\subset\mathbb P^{n+1}$ be a smooth complex projective hypersurface and let $A\to X$ be an ample line bundle. Then there exists a positive integer $\delta_n$ such that
$$
H^0(X,E_{k,m}T^*_X\otimes A^{-1})\ne 0,\quad k\ge n,
$$
provided that $\deg(X)\ge\delta_n$ and $m$ is large enough.
\end{theorem}

In other words, on every smooth $n$-dimensional complex projective hypersurface of sufficiently high degree, there exist global invariant jet differentials of order $n$ vanishing on an ample divisor, and every entire curve must satisfy the corresponding differential equation.

Unfortunately, the lower bound for the degree of $X$ is effective just theoretically,  and one can compute explicit values just for low dimensions (see \cite{Div08}). Nevertheless,  the result is sharp as far as the order $k$ of jets is concerned since, by a theorem of \cite{Div08}, there are no jet differentials of order $k<n$ on a smooth projective hypersurface of dimension $n$.

We also prove a logarithmic version of the above theorem.

\begin{theorem}\label{logexistence}
Let $D\subset\mathbb P^n$ be a smooth irreducible divisor and let $A\to \mathbb P^n$ be an ample line bundle. Then there exists a positive integer $\delta_n$ such that
$$
H^0(\mathbb P^n,E_{k,m}T^*_{\mathbb P^n}\langle D\rangle\otimes A^{-1})\ne 0,\quad k\ge n,
$$
provided that $\deg(D)\ge\delta_n$ and $m$ is large enough.

Moreover, we have the effective lower bounds for the degree $\delta$ of $D$ as shown in Table \ref{logeffres} (depending on the values of $n$ and $k$).
\end{theorem}

\begin{table}
\caption{Effective lower bound for degree $\delta$.}\label{logeffres}
\centering
\vspace{0,5cm}
\begin{tabular}{*{6}{c}}
\hline 
 & \multicolumn{5}{c}{$\mathbf{k}$} \\ \cline{2-6} 
$\mathbf{n}$ & 1 & 2 & 3 & 4 & 5 \\ \hline
2  &  & 15 & 14 & 14 & 14 \\
3 &  &  & 75 & 67 & 67 \\
4 &  &  &  & 306 & 280\\
5 &  &  &  &  & 1154\\ \hline
\end{tabular}
\end{table} 

In the statement above, the bundle $E_{k,m}T^*_{\mathbb P^n}\langle D\rangle$, is the vector bundle of logarithmic invariant jet differentials, introduced in the general setting of logarithmic varieties by Dethloff and Lu in \cite{D-L01}.

Finally, we would like to stress that our proof is based on the algebraic version of holomorphic Morse inequalities of \cite{Trapani95}, and so we deal directly with the dimension of the space of global sections$\,$: we are able in this way to skip entirely the arduous study of the higher cohomology and of the graded bundle associated to $E_{k,m}T^*_X$.

\subsection*{Acknowledgements}
We are indebted to Prof.\ Stefano Trapani, who suggested to us the main ideas in Lemma \ref{cruc1} and Lemma \ref{cruc2}. 

We would like to thank also Prof. Jean-Pierre Demailly for a careful reading of the present manuscript, and for many useful suggestions and comments.

Finally, thanks to Gianluca Pacienza who suggested us to try to extend our results to the logarithmic case.

\section{Notations and preliminary material}

Let $X$ be a compact complex manifold and $V\subset T_X$ a holomorphic (non necessarily integrable) subbundle of the tangent bundle of $X$.

\subsection{Invariant jet differentials}

The bundle $J_kV\to X$ is the bundle of $k$-jets of holomorphic curves $f\colon(\mathbb C,0)\to X$ which are tangent to $V$, i.e., such that $f'(t)\in V_{f(t)}$ for all $t$ in a neighbourhood of $0$, together with the projection map $f\mapsto f(0)$ onto $X$.

Let $\mathbb G_k$ be the group of germs of $k$-jets of biholomorphisms of $(\mathbb C,0)$, that is, the group of germs of biholomorphic maps
$$
t\mapsto\varphi(t)=a_1\, t+ a_2\, t^2+\cdots+a_k\,t^k,\quad a_1\in\mathbb C^*,a_j\in\mathbb C,j\ge 2,
$$
in which the composition law is taken modulo terms $t^j$ of degree $j>k$. Then $\mathbb G_k$ admits a natural fiberwise right action on $J_kV$ consisting of reparametrizing $k$-jets of curves by a biholomorphic change of parameter. 

Next, we define the bundle of Demailly-Semple jet differentials (or invariant jet differentials).

\begin{definition}[\cite{Demailly95}]
The vector bundle of \emph{invariant jet differentials of order $k$ and degree $m$} is the bundle $E_{k,m}V^*\to X$ of polynomial differential operators $Q(f',f'',\dots,f^{(k)})$ over the fibers of $J_kV$, which are invariant under arbitrary changes of parametrization, \emph{i.e.} for every $\varphi\in\mathbb G_k$
$$
Q((f\circ\varphi)',(f\circ\varphi)'',\dots,(f\circ\varphi)^{(k)})=\varphi'(0)^m\,Q(f',f'',\dots,f^{(k)}).
$$
\end{definition}

\subsection{Projectivized jet bundles}

Here is the construction of the tower of projectivized bundles which provides a (relative) smooth compactification of $J^{\text{reg}}_kV/\mathbb G_k$, where $J^{\text{reg}}_kV$ is the bundle of regular $k$-jets tangent to $V$, that is $k$-jets such that $f'(0)\ne 0$.

Let $\dim X=n$ and $\rank V=r$. With $(X,V)$ we associate another \lq\lq directed\rq\rq{} manifold $(\widetilde X,\widetilde V)$ where $\widetilde X=P(V)$ is the projectivized bundle of lines of $V$, $\pi\colon\widetilde X\to X$ is the natural projection and $\widetilde V$ is the subbundle of $T_{\widetilde X}$ defined fiberwise as
$$
\widetilde V_{(x_0,[v_0])}\overset{\text{def}}=\{\xi\in T_{\widetilde X,(x_0,[v_0])}\mid\pi_*\xi\in\mathbb C.v_0\},
$$
$x_0\in X$ and $v_0\in T_{X,x_0}\setminus\{0\}$. 
We also have a \lq\lq lifting\rq\rq{} operator which assigns to a germ of holomorphic curve $f\colon(\mathbb C,0)\to X$ tangent to $V$ a germ of holomorphic curve $\widetilde f\colon(\mathbb C,0)\to\widetilde X$ tangent to $\widetilde V$ in such a way that $\widetilde f(t)=(f(t),[f'(t)])$.

To construct the projectivized $k$-jet bundle we simply set inductively $(X_0,V_0)=(X,V)$ and $(X_k,V_k)=(\widetilde X_{k-1},\widetilde V_{k-1})$. Of course, we have for each $k>0$ a tautological line bundle $\mathcal O_{X_k}(-1)\to X_k$ and a natural projection $\pi_k\colon X_k\to X_{k-1}$. We shall call $\pi_{j,k}$ the composition of the projections $\pi_{j+1}\circ\cdots\circ\pi_{k}$, so that the total projection is given by $\pi_{0,k}\colon X_k\to X$.
For each $k>0$, we have short exact sequences
\begin{equation}\label{ses1}
0\to T_{X_k/X_{k-1}}\to V_k\to\mathcal O_{X_k}(-1)\to 0,
\end{equation}
\begin{equation}\label{ses2}
0\to\mathcal O_{X_k}\to\pi_k^*V_{k-1}\otimes\mathcal O_{X_k}(1)\to T_{X_k/X_{k-1}}\to 0,
\end{equation}
where $T_{X_k/X_{k-1}}=\ker(\pi_k)_*$ is the vertical tangent bundle relative to $\pi_k$ and $\rank V_k=r$, $\dim X_k=n+k(r-1)$. Here, we also have an inductively defined $k$-lifting for germs of holomorphic curves such that $f_{[k]}\colon(\mathbb C,0)\to X_k$ is obtained as $f_{[k]}=\widetilde f_{[k-1]}$.

The following theorem is the link between these projectivized bundles and jet differentials.

\begin{theorem}[\cite{Demailly95}]\label{di}
Suppose that $\rank V\ge 2$. The quotient space $J_k^{\text{\rm reg}}V/\mathbb G_k$ has the structure of a locally trivial bundle over $X$, and there is a holomorphic embedding $J_k^{\text{\rm reg}}V/\mathbb G_k\hookrightarrow X_k$ over $X$, which identifies $J_k^{\text{\rm reg}}V/\mathbb G_k$ with $X_k^{\text{\rm reg}}$, that is the set of point in $X_k$ on the form $f_{[k]}(0)$ for some non singular $k$-jet $f$. In other words $X_k$ is a relative compactification of $J_k^{\text{\rm reg}}V/\mathbb G_k$ over $X$.

Moreover, we have the direct image formula
$$
(\pi_{0,k})_*\mathcal O_{X_k}(m)=\mathcal O(E_{k,m}V^*).
$$
\end{theorem}

Next, here is the link between the theory of hyperbolicity and invariant jet differentials.

\begin{theorem}[\cite{G-G79},\cite{Demailly95}]
Assume that there exist integers $k,m>0$ and an ample line bundle $A\to X$ such that 
$$
H^0(X_k,\mathcal O_{X_k}(m)\otimes\pi_{0,k}^*A^{-1})\simeq H^0(X,E_{k,m}V^*\otimes A^{-1})
$$
has non zero sections $\sigma_1,\dots,\sigma_N$ and let $Z\subset X_k$ be the base locus of these sections. Then every entire holomorphic curve $f\colon\mathbb C\to X$ tangent to $V$ is such that $f_{[k]}(\mathbb C)\subset Z$. In other words, for every global $\mathbb G_k$-invariant differential equation $P$ vanishing on an ample divisor, every entire holomorphic curve $f$ must satisfy the algebraic differential equation $P(f)=0$.
\end{theorem}

\subsection{Cohomology ring of $X_k$}

Denote by $c_\bullet(E)$ the total Chern class of a vector bundle $E$. The short exact sequences (\ref{ses1}) and (\ref{ses2}) give, for each $k>0$, the following formulae:
$$
c_\bullet(V_k)=c_\bullet(T_{X_k/X_{k-1}}) c_\bullet(\mathcal O_{X_k}(-1))
$$
and
$$
c_\bullet(\pi_k^*V_{k-1}\otimes\mathcal O_{X_k}(1))=c_\bullet(T_{X_k/X_{k-1}}),
$$
so that
\begin{equation}\label{chern1}
c_\bullet(V_k)=c_\bullet(\mathcal O_{X_k}(-1))c_\bullet(\pi_k^*V_{k-1}\otimes\mathcal O_{X_k}(1)).
\end{equation}
Let us call $u_j=c_1(\mathcal O_{X_j}(1))$ and $c_l^{[j]}=c_l(V_j)$. With these notations, (\ref{chern1}) becomes
\begin{equation}\label{chern2}
c_l^{[k]}=\sum_{s=0}^l\left[\binom{n-s}{l-s}-\binom{n-s}{l-s-1}\right]u_k^{l-s}\cdot\pi_k^*c_s^{[k-1]},\quad 1\le l\le r.
\end{equation}
Since $X_j$ is the projectivized bundle of line of $V_{j-1}$, we also have the polynomial relations
\begin{equation}\label{chern3}
u_j^r+\pi_j^*c_1^{[j-1]}\cdot u_j^{r-1}+\cdots+\pi_j^*c_{r-1}^{[j-1]}\cdot u_j+\pi_j^*c_{r}^{[j-1]}=0,\quad 1\le j\le k.
\end{equation}
After all, the cohomology ring of $X_k$ is defined in terms of generators and relations as the polynomial algebra $H^\bullet(X)[u_1,\dots,u_k]$ with the relations (\ref{chern3}) in which, of course, utilizing inductively (\ref{chern2}), we have that $c_l^{[j]}$ is a polynomial with integral coefficients in the variables $u_1,\dots,u_j,c_1(V),\dots,c_l(V)$.

In particular, for the first Chern class of $V_k$, we obtain the very simple expression
\begin{equation}\label{c1}
c_1^{[k]}=\pi_{0,k}^*c_1(V)+(r-1)\sum_{s=1}^k\pi_{s,k}^*u_s.
\end{equation}

\subsection{Algebraic holomorphic Morse inequalities}

Let $L\to X$ be a holomorphic line bundle over a compact K\"ahler manifold of dimension $n$ and $E\to X$ a holomorphic vector bundle of rank $r$. Suppose that $L$ can be written as the difference of two nef line bundles, say $L=F\otimes G^{-1}$, with $F,G\to X$ numerically effective. Then we have the following asymptotic estimate for the partial alternating sum of the dimension of cohomology groups of powers of $L$ with values in $E$.

\begin{theorem}[\cite{Demailly00}]
With the previous notation, we have $($strong algebraic holomorphic Morse inequalities$):$
$$
\sum_{ j=0}^q(-1)^{q-j}h^j(X,L^{\otimes m}\otimes E)\le r\frac{m^n}{n!}\sum_{j=0}^q(-1)^{q-j}\binom nj F^{n-j}\cdot G^j+o(m^n).
$$
In particular {\rm \cite{Trapani95}}, $L^{\otimes m}\otimes E$ has a global section for $m$ large as soon as $F^n-n\,F^{n-1}\cdot G>0$.
\end{theorem}

\subsection{A vanishing theorem} Let $X\subset\mathbb P^N$ be a smooth complete intersection of dimension $\dim X=n$. In \cite{Div08}, we proved the following vanishing theorem for the space of global sections of invariant jet differentials:

\begin{theorem}[\cite{Div08}]\label{vanish} 
Let $X\subset\mathbb P^{N}$ be a smooth complete intersection of dimension $\dim X=n$. Then 
$$
H^0(X,E_{k,m}T^*_X)=0
$$
for all $m\ge 1$ and $1\le k< n/(N-n)$. In particular, if $X$ is a smooth projective hypersurface, then
$$
H^0(X,E_{k,m}T^*_X)=0
$$
for all $m\ge 1$ and $1\le k\le n-1$.
\end{theorem}

\section{Proof of Theorem \ref{existence}}

The idea of the proof is to apply the algebraic holomorphic Morse inequalities to a particular relatively nef line bundle over $X_n$ which admits a nontrivial morphism to (a power of) $\mathcal O_{X_n}(1)$ and then to conclude by the direct image argument of Theorem \ref{di}. 

\subsection{Sufficient conditions for relative nefness}

By definition, there is a canonical injection $\mathcal O_{X_k}(-1)\hookrightarrow\pi_k^*V_{k-1}$ and a composition with the differential of the projection $(\pi_k)_*$ yields, for all $k\ge 2$, a canonical line bundle morphism
$$
\mathcal O_{X_k}(-1)\hookrightarrow\pi_k^*V_{k-1}\to\pi_k^*\mathcal O_{X_{k-1}}(-1),
$$
which admits precisely $D_k\overset{\text{def}}=P(T_{X_{k-1}/X_{k-2}})\subset P(V_{k-1})=X_k$ as its zero divisor. Hence, we find
\begin{equation}\label{mor}
\mathcal O_{X_k}(1)=\pi_k^*\mathcal O_{X_{k-1}}(1)\otimes\mathcal O(D_k).
\end{equation}
Now, for $\mathbf a=(a_1,\dots,a_k)\in\mathbb Z^k$, define a line bundle $\mathcal O_{X_k}(\mathbf a)$ on $X_k$ as
$$
\mathcal O_{X_k}(\mathbf a)=\pi_{1,k}^*\mathcal O_{X_1}(a_1)\otimes\pi_{2,k}^*\mathcal O_{X_2}(a_2)\otimes\cdots\otimes\mathcal O_{X_k}(a_k).
$$
By (\ref{mor}), we have
$$
\pi_{j,k}^*\mathcal O_{X_j}(1)=\mathcal O_{X_k}(1)\otimes\mathcal O_{X_k}(-\pi_{j+1,k}^*D_{j+1}-\cdots-D_k),
$$
thus by putting $D_j^\star=\pi_{j+1,k}^*D_{j+1}$ for $j=1,\dots,k-1$ and $D_k^\star=0$, we have an identity
$$
\begin{aligned}
& \mathcal O_{X_k}(\mathbf a)=\mathcal O_{X_k}(b_k)\otimes\mathcal O_{X_k}(-\mathbf b\cdot D^\star),\quad\text{where} \\
& \mathbf b=(b_1,\dots,b_k)\in\mathbb Z^k,\quad b_j=a_1+\cdots+a_j, \\
& \mathbf b\cdot D^\star=\sum_{j=1}^{k-1}b_j\,\pi_{j+1,k}^*D_{j+1}.
\end{aligned}
$$
In particular, if $\mathbf b\in\mathbb N^k$, that is if $a_1+\cdots+a_j\ge 0$, we get a nontrivial morphism
$$
\mathcal O_{X_k}(\mathbf a)=\mathcal O_{X_k}(b_k)\otimes\mathcal O_{X_k}(-\mathbf b\cdot D^\star)\to\mathcal O_{X_k}(b_k).
$$
We then have the following:

\begin{proposition}[\cite{Demailly95}]
Let $\mathbf a=(a_1,\dots,a_k)\in\mathbb N^k$ and $m=a_1+\cdots+a_k$. 
\begin{itemize}
\item We have the direct image formula
$$
(\pi_{0,k})_*\mathcal O_{X_k}(\mathbf a)\simeq\mathcal O(\overline F^{\mathbf a}E_{k,m}V^*)\subset\mathcal O(E_{k,m}V^*)
$$
where $\overline F^{\mathbf a}E_{k,m}V^*$ is the subbundle of polynomials $Q(f',\dots,f^{(k)})$ of $E_{k,m}V^*$ involving only monomials $(f^{(\bullet)})^\ell$ such that
$$
\ell_{s+1}+2\ell_{s+2}+\cdots+(k-s)\ell_{k}\le a_{s+1}+\cdots+a_k
$$
for all $s=0,\dots,k-1$.
\item If 
\begin{equation}\label{relnef}
a_1\ge 3a_2,\dots,a_{k-2}\ge 3a_{k-1}\quad\text{and}\quad a_{k-1}\ge 2a_k> 0,
\end{equation}
the line bundle $\mathcal O_{X_k}(\mathbf a)$ is relatively nef over $X$.
\end{itemize}
\end{proposition}

From now on, we will set in the absolute case, that is $V=T_X$. So, let $X\subset\mathbb P^{n+1}$ be a smooth complex hypersurface of degree $\deg X=d$. 

Now, for the projective hypersurfaces case, it is always possible to express $\mathcal O_{X_k}(\mathbf a)$ as the difference of two globally nef line bundles, provided condition (\ref{relnef}) is satisfied. We prove this fact in the next:

\begin{proposition}
Let $X\subset\mathbb P^{n+1}$ be a smooth projective hypersurface and $\mathcal O_X(1)$ be the hyperplane divisor on $X$. If condition (\ref{relnef}) holds, then $\mathcal O_{X_k}(\mathbf a)\otimes\pi_{0,k}^*\mathcal O_X(\ell)$ is nef provided that $\ell\ge 2|\mathbf a|$, where $|\mathbf a|=a_1+\cdots +a_k$.

In particular $\mathcal O_{X_k}(\mathbf a)=\bigl(\mathcal O_{X_k}(\mathbf a)\otimes\pi_{0,k}^*\mathcal O_X(2|\mathbf a|)\bigr)\otimes\pi_{0,k}^*\mathcal O_X(-2|\mathbf a|)$ and both $\mathcal O_{X_k}(\mathbf a)\otimes\pi_{0,k}^*\mathcal O_X(2|\mathbf a|)$ and $\pi_{0,k}^*\mathcal O_X(2|\mathbf a|)$ are nef.
\end{proposition}

\begin{proof}
In \cite{Div08} we proved that the line bundle 
$$
\mathcal O_{X_k}(2\cdot 3^{k-2},2\cdot 3^{k-3},\dots,6,2,1)\otimes\pi_{0,k}^*\mathcal O_X(\ell)
$$ 
is nef as soon as $\ell\ge 2\cdot(1+2+6+\cdots+2\cdot 3^{k-2})=2\cdot 3^{k-1}$. Now we take $\mathbf a=(a_1,\dots,a_k)\in\mathbb N^{k}$ such that $a_1\ge 3a_2,\dots,a_{k-2}\ge 3a_{k-1},a_{k-1}\ge 2a_k>0$ and we proceed by induction, the case $k=1$ being obvious. Write
$$
\begin{aligned}
\mathcal O_{X_k}&(a_1,a_2,\dots,a_k)\otimes\pi_{0,k}^*\mathcal O_X(2\cdot(a_1+\cdots+a_k)) \\
&=\bigr(\mathcal O_{X_k}(2\cdot 3^{k-2},\dots,6,2,1)\otimes\pi_{0,k}^*\mathcal O_X(2\cdot 3^{k-1})\bigl)^{\otimes a_k} \\
&\quad\otimes\pi_{k}^*\biggl(\mathcal O_{X_{k-1}}(a_1-2\cdot 3^{k-2}a_k,\dots,a_{k-2}-6a_k,a_{k-1}-2a_k) \\
&\quad\otimes\pi_{0,k-1}^*\mathcal O_X\bigl(2\cdot(a_1+\dots+a_k-3^{k-1}a_k)\bigr)\biggr).
\end{aligned}
$$
Therefore, we have to prove that 
$$
\begin{aligned}
&\mathcal O_{X_{k-1}}(a_1-2\cdot 3^{k-2}a_k,\dots,a_{k-2}-6a_k,a_{k-1}-2a_k) \\
&\quad\quad\otimes\pi_{0,k-1}^*\mathcal O_X\bigl(2\cdot(a_1+\dots+a_k-3^{k-1}a_k)\bigr)
\end{aligned}
$$
is nef.
Our chain of inequalities gives, for $1\le j\le k-2$, $a_j\ge 3^{k-j-1}a_k$ and $a_{k-1}\ge 2a_k$. Thus, condition (\ref{relnef}) is satisfied by the weights of
$$
\mathcal O_{X_{k-1}}(a_1-2\cdot 3^{k-2}a_k,\dots,a_{k-2}-6a_k,a_{k-1}-2a_k)
$$
and $2\cdot(a_1+\dots+a_k-3^{k-1}a_k)$ is exactly twice the sum of these weights. 
\end{proof}

\begin{remark}\label{rem1}
At this point it should be clear that  to prove our Theorem, it is sufficient to show the existence of an $n$-tuple $(a_1,\dots,a_n)$ satisfying condition (\ref{relnef}) and such that
\begin{equation}\label{*}
\begin{aligned}
\bigl(\mathcal O_{X_n}&(\mathbf a)\otimes\pi_{0,n}^*\mathcal O_X(2|\mathbf a|)\bigr)^{n^2} \\
&-n^2\bigl(\mathcal O_{X_n}(\mathbf a)\otimes\pi_{0,n}^*\mathcal O_X(2|\mathbf a|)\bigr)^{n^2-1}\cdot\pi_{0,n}^*\mathcal O_X(2|\mathbf a|)>0
\end{aligned}
\end{equation}
for $d=\deg X$ large enough, where $n^2=n+n(n-1)=\dim X_n$.

In fact, this would show the bigness of $\mathcal O_{X_n}(\mathbf a)\hookrightarrow\mathcal O_{X_n}(|\mathbf a|)$ and so the bigness of $\mathcal O_{X_n}(1)$.
\end{remark}

\subsection{Evaluation in terms of the degree}

For $X\subset\mathbb P^{n+1}$ a smooth projective hypersurface of degree $\deg X=d$, standard arguments involving the Euler exact sequence, show that
$$
c_j(X)=c_j(T_X)=h^j\bigl((-1)^jd^j+o(d^j)\bigr),\quad j=1,\dots,n,
$$
where $h\in H^2(X,\mathbb Z)$ is the hyperplane class and $o(d^j)$ is a polynomial in $d$ of degree at most $j-1$.

\begin{proposition}\label{nef}
The quantities
$$
\begin{aligned}
&\bigl(\mathcal O_{X_k}(\mathbf a)\otimes\pi_{0,k}^*\mathcal O_X(2|\mathbf a|)\bigr)^{n+k(n-1)} \\
& -[n+k(n-1)]\bigl(\mathcal O_{X_k}(\mathbf a)\otimes\pi_{0,k}^*\mathcal O_X(2|\mathbf a|)\bigr)^{n+k(n-1)-1}\cdot\pi_{0,k}^*\mathcal O_X(2|\mathbf a|)
\end{aligned}
$$
and
$$
\mathcal O_{X_k}(\mathbf a)^{n+k(n-1)}
$$
are both polynomials in the variable $d$ with coefficients in $\mathbb Z[a_1,\dots,a_k]$ of degree at most $n+1$ and the coefficients of $d^{n+1}$ of the two expressions are equal.

Moreover this coefficient is a homogeneous polynomial in $a_1,\dots,a_k$ of degree $n+k(n-1)$ or identically zero.
\end{proposition}

\begin{proof}
Set $\mathcal F_k(\mathbf a)=\mathcal O_{X_k}(\mathbf a)\otimes\pi_{0,k}^*\mathcal O_X(2|\mathbf a|)$ and $\mathcal G_k(\mathbf a)=\pi_{0,k}^*\mathcal O_X(2|\mathbf a|)$. Then we have
$$
\begin{aligned}
\mathcal F_k(\mathbf a)&^{n+k(n-1)}+[n+k(n-1)]\mathcal F_k(\mathbf a)^{n+k(n-1)-1}\cdot\mathcal G_k(\mathbf a)\\
&=\mathcal O_{X_k}(\mathbf a)^{n+k(n-1)}+\text{\rm terms which have $\mathcal G_k(\mathbf a)$ as a factor}.
\end{aligned}
$$
Now we use relations (\ref{chern2}) and (\ref{chern3}) to observe that
$$
\mathcal O_{X_k}(\mathbf a)^{n+k(n-1)}=\sum_{j_1+2j_2+\cdots+nj_n=n}P_{j_1\cdots j_n}^{[k]}(\mathbf a)\,c_1(X)^{j_1}\cdots c_n(X)^{j_n},
$$
where the $P_{j_1\cdots j_n}^{[k]}(\mathbf a)$'s are homogeneous polynomial of degree $n+k(n-1)$ in the variables $a_1,\dots,a_k$ (or possibly identically zero). Thus, substituting the $c_j(X)$'s with their expression in terms of the degree, we get
$$
\mathcal O_{X_k}(\mathbf a)^{n+k(n-1)}=(-1)^n\left(\sum_{j_1+2j_2+\cdots+nj_n=n}P_{j_1\cdots j_n}^{[k]}(\mathbf a)\right)d^{n+1}+o(d^{n+1}),
$$
since $h^n=d$. On the other hand, using relations (\ref{chern2}) and (\ref{chern3}) on terms which have $\mathcal G_k(\mathbf a)$ as a factor, gives something of the form
$$
\sum_{\genfrac{}{}{0pt}{}{j_1+2j_2+\cdots+nj_n+i=n}{i>0}} Q_{j_1\cdots j_ni}^{[k]}(\mathbf a)\,h^i\cdot c_1(X)^{j_1}\cdots c_n(X)^{j_n},
$$
since $c_1(\mathcal G_k(\mathbf a))=2|\mathbf a|h$ and $\mathcal G_k(\mathbf a)$ is always a factor. Substituting the $c_j(X)$'s with their expression in terms of the degree, we get here
$$
h^i\cdot c_1(X)^{j_1}\cdots c_n(X)^{j_n}=(-1)^{j_1+\cdots+j_n}\underbrace{h^n}_{=d}\cdot d^{j_1+\dots+j_n}=o(d^{n+1}).
$$
\end{proof}

We need a lemma.

\begin{lemma}\label{Zar}
Let $\mathfrak C\subset\mathbb R^k$ be a cone with nonempty interior. Let $\mathbb Z^k\subset\mathbb R^k$ be the canonical lattice in $\mathbb R^k$. Then $\mathbb Z^k\cap\mathfrak C$ is Zariski dense in $\mathbb R^k$.
\end{lemma}

\begin{proof}
Since $\mathfrak C$ is a cone with nonempty interior, it contains cubes of arbitrary large edges, so $\mathbb Z^k\cap\mathfrak C$ contains a product of integer intervals $\prod[\alpha_i,\beta_i]$ with $\beta_i-\alpha_i>N$. By using induction on dimension, this implies that a polynomial $P$ of degree at most $N$ vanishing on $\mathbb Z^k\cap\mathfrak C$ must be identically zero. As $N$ can be taken arbitrary large, we conclude that  $\mathbb Z^k\cap\mathfrak C$ is Zariski dense.\end{proof}

\begin{corollary}\label{bigness1}
If the top self-intersection $\mathcal O_{X_k}(\mathbf a)^{n+k(n-1)}$ has degree exac\-tly equal to $n+1$ in $d$ for some choice of $\mathbf a$, then $\mathcal O_{X_k}(m)\otimes\pi_{0,k}^*A^{-1}$ has a global section for all line bundle $A\to X$ and for all $d,m$ sufficiently large.
\end{corollary}

\begin{proof}
The real $k$-tuples which satisfy condition (\ref{relnef}), form a cone with non-empty interior in $\mathbb R^k$.
Thus, by Lemma \ref{Zar}, there exists an integral $\mathbf a'$ satisfying condition (\ref{relnef}) and such that $\mathcal O_{X_k}(\mathbf a')^{n+k(n-1)}$ has degree exactly $n+1$ in $d$. For reasons similar to those in the proof of Proposition \ref{nef}, the coefficient of degree $n+1$ in $d$ of $\mathcal O_{X_k}(\mathbf a')^{n+k(n-1)}$ and $\bigl(\mathcal O_{X_k}(\mathbf a')\otimes\pi_{0,k}^*\mathcal O_X(2|\mathbf a'|)\bigr)^{n+k(n-1)}$ are the same; the second one being nef, this coefficient must be positive.

Now, by Proposition \ref{nef}, this coefficient is the same as the coefficient of degree $n+1$ in $d$ of
$$
\begin{aligned}
&\bigl(\mathcal O_{X_k}(\mathbf a')\otimes\pi_{0,k}^*\mathcal O_X(2|\mathbf a'|)\bigr)^{n+k(n-1)} \\
& -[n+k(n-1)]\bigl(\mathcal O_{X_k}(\mathbf a')\otimes\pi_{0,k}^*\mathcal O_X(2|\mathbf a'|)\bigr)^{n+k(n-1)-1}\cdot\pi_{0,k}^*\mathcal O_X(2|\mathbf a'|).
\end{aligned}
$$
But then this last quantity is positive for $d$ large enough, and the Corollary follows by an application of algebraic holomorphic Morse inequalities.
\end{proof}

\begin{corollary}\label{bigness2}
For $k<n$, the coefficient of $d^{n+1}$ in the expression of 
$$
\mathcal O_{X_k}(\mathbf a)^{n+k(n-1)}
$$ 
is identically zero.
\end{corollary}

\begin{proof}
Otherwise, we would have global sections of $\mathcal O_{X_k}(m)$ for $m$ large and $k<n$, which is impossible by Theorem \ref{vanish}.
\end{proof}

\subsection{Bigness of $\mathcal O_{X_n}(1)$} Thanks to the results of the previous subsection, to show the existence of a global section of $\mathcal O_{X_n}(m)\otimes\pi_{0,n}^*A^{-1}$ for $m$ and $d$ large, we just need to show that $\mathcal O_{X_n}(\mathbf a)^{n^2}$ has degree exactly $n+1$ in $d$ for some $n$-tuple $(a_1,\dots,a_n)$.

The multinomial theorem gives
$$
\begin{aligned}
(a_1\pi_{1,k}^*u_1&+\cdots+a_ku_k)^{n+k(n-1)}\\ &=\sum_{j_1+\dots+j_k=n+k(n-1)}\frac{(n+k(n-1))!}{j_1!\cdots j_k!}\,a_1^{j_1}\cdots a_k^{j_k}\, \pi_{1,k}^*u_1^{j_1}\cdots u_k^{j_k}.
\end{aligned}
$$
We need two lemmas.

\begin{lemma}\label{cruc1}
The coefficient of degree $n+1$ in $d$ of the two following intersections is zero:
\begin{itemize}
\item[$\bullet$] $\pi_{1,k}^*u_1^{j_1}\cdot\pi_{2,k}^*u_2^{j_2}\cdots u_k^{j_k}$ for all $1\le k\le n-1$ and $j_1+\dots+j_k=n+k(n-1)$
\item[$\bullet$] $\pi_{1,n-i-1}^*u_1^{j_1}\cdot\pi_{2,n-i-1}^*u_2^{j_2}\cdots u_{n-i-1}^{j_{n-i-1}}\cdot\pi_{0,n-i-1}^*c_1(X)^i$ for all $1\le i\le n-2$ and $j_1+\dots+j_{n-i-1}=(n-i-1)n+1$. 
\end{itemize}
\end{lemma}

\begin{proof}
The first statement is straightforward. By Corollary \ref{bigness2}, we know that the coefficient of degree $n+1$ in $d$ (once the expression is reduced in term of the degree of the hypersurface) of  $(a_1\pi_{1,k}^*u_1+\cdots+a_ku_k)^{n+k(n-1)}$ must be identically zero for $k<n$. If this first part of the lemma fails to be true for some $(j_1,\dots,j_k)$, then this leading coefficient would contain al least a monomial, namely $a_1^{j_1}\cdots a_k^{j_k}$ and thus it would not be identically zero.



For the second statement we proceed by induction on $i$. Let us start with $i=1$. By the first part of the present lemma, we have that 
$$
\pi_{1,n-1}^*u_1^{j_1}\cdot\pi_{2,n-1}^*u_2^{j_2}\cdots\pi_{n-1}^*u_{n-2}^{j_{n-2}}\cdot u_{n-1}^n=o(d^{n+1}).
$$ 
On the other hand, relation (\ref{chern3}) gives
$$
\begin{aligned}
\pi_{1,n-1}^*&u_1^{j_1}\cdot\pi_{2,n-1}^*u_2^{j_2}\cdots\pi_{n-1}^*u_{n-2}^{j_{n-2}}\cdot u_{n-1}^n \\
&=\pi_{1,n-1}^*u_1^{j_1}\cdot\pi_{2,n-1}^*u_2^{j_2}\cdots\pi_{n-1}^*u_{n-2}^{j_{n-2}}\\
&\quad \cdot\bigl(-\pi_{n-1}^*c_1^{[n-2]}\cdot u_{n-1}^{n-1}-\cdots-\pi_{n-1}^*c_{n-1}^{[n-2]}\cdot u_{n-1}-\pi_{n-1}^*c_{n}^{[n-2]}\bigr) \\
&=-\pi_{1,n-1}^*u_1^{j_1}\cdot\pi_{2,n-1}^*u_2^{j_2}\cdots\pi_{n-1}^*u_{n-2}^{j_{n-2}}\cdot\pi_{n-1}^*c_1^{[n-2]}\cdot u_{n-1}^{n-1}
\end{aligned}
$$
and the second equality is true for degree reasons: 
$$
u_1^{j_1}\cdot u_2^{j_2}\cdots u_{n-2}^{j_{n-2}}\cdot c_l^{[n-2]},\quad l=2,\dots,n,
$$ 
\lq\lq lives\rq\rq{} on $X_{n-2}$ and has total degree $n+(n-2)(n-1)-1+l$ which is strictly greater than $n+(n-2)(n-1)=\dim X_{n-2}$, so that $u_1^{j_1}\cdot u_2^{j_2}\cdots u_{n-2}^{j_{n-2}}\cdot c_l^{[n-2]}=0$. Now, we use relation (\ref{c1}) and obtain in this way
$$
\begin{aligned}
\pi_{1,n-1}^*u_1^{j_1}&\cdot\pi_{2,n-1}^*u_2^{j_2}\cdots\pi_{n-1}^*u_{n-2}^{j_{n-2}}\cdot u_{n-1}^n\\
&=-\pi_{1,n-1}^*u_1^{j_1}\cdot\pi_{2,n-1}^*u_2^{j_2}\cdots\pi_{n-1}^*u_{n-2}^{j_{n-2}}\cdot u_{n-1}^{n-1}\\
&\quad\cdot\biggl(\pi_{0,n-1}^*c_1(X)+(n-1)\sum_{s=1}^{n-2}\pi_{s,n-1}^*u_s\biggr)\\
&=-\pi_{1,n-1}^*u_1^{j_1}\cdot\pi_{2,n-1}^*u_2^{j_2}\cdots\pi_{n-1}^*u_{n-2}^{j_{n-2}}\cdot u_{n-1}^{n-1}\cdot\pi_{0,n-1}^*c_1(X)\\
&\quad-(n-1)u_{n-1}^{n-1}\cdot\sum_{s=1}^{n-2}\pi_{1,n-1}^*u_1^{j_1}\cdots\pi_{s,n-1}^*u_s^{j_s+1}\cdots\pi_{n-1}^*u_{n-2}^{j_{n-2}}.
\end{aligned}
$$ 
An integration along the fibers of $X_{n-1}\to X_{n-2}$ then gives
$$
\begin{aligned}
\pi_{1,n-2}^*u_1^{j_1}&\cdot\pi_{2,n-2}^*u_2^{j_2}\cdots u_{n-2}^{j_{n-2}}\cdot\pi_{0,n-2}^*c_1(X)\\
&=-(n-1)\cdot\sum_{s=1}^{n-2}\underbrace{\pi_{1,n-2}^*u_1^{j_1}\cdots\pi_{s,n-2}^*u_s^{j_s+1}\cdots u_{n-2}^{j_{n-2}}}_{=o(d^{n+1})\,\,\text{\rm by the first part of the lemma}}\\
&\quad\quad+o(d^{n+1})
\end{aligned}
$$
and so $\pi_{1,n-2}^*u_1^{j_1}\cdot\pi_{2,n-2}^*u_2^{j_2}\cdots u_{n-2}^{j_{n-2}}\cdot\pi_{0,n-2}^*c_1(X)=o(d^{n+1})$.

To complete the proof, observe that -- as before -- relations (\ref{chern3}) and (\ref{c1}) together with a completely similar degree argument give 
$$
\begin{aligned}
\pi_{1,n-i}^*&u_1^{j_1}\cdot\pi_{2,n-i}^*u_2^{j_2}\cdots u_{n-i}^{j_{n-i-1}}\cdot\pi_{0,n-i}^*c_1(X)^i\cdot u_{n-i}^n \\
&=-\pi_{1,n-i}^*u_1^{j_1}\cdot\pi_{2,n-i}^*u_2^{j_2}\cdots\pi_{n-i}^*u_{n-i-1}^{j_{n-i-1}}\cdot u_{n-i}^{n-1}\cdot\pi_{0,n-i}^*c_1(X)^{i+1}\\
&\quad-(n-1)u_{n-i}^{n-1}\cdot\sum_{s=1}^{n-i-1}\pi_{1,n-i}^*u_1^{j_1}\cdots\pi_{s,n-i}^*u_s^{j_s+1}\cdots\pi_{n-i}^*u_{n-i-1}^{j_{n-i-1}}.
\end{aligned}
$$
But 
$$
\pi_{1,n-i}^*u_1^{j_1}\cdot\pi_{2,n-i}^*u_2^{j_2}\cdots u_{n-i}^{j_{n-i-1}}\cdot\pi_{0,n-i}^*c_1(X)^i\cdot u_{n-i}^n=o(d^{n+1})
$$ 
by induction, and 
$$
\pi_{1,n-i}^*u_1^{j_1}\cdots\pi_{s,n-i}^*u_s^{j_s+1}\cdots\pi_{n-i}^*u_{n-i-1}^{j_{n-i-1}}=o(d^{n+1}),
$$ 
$1\le s\le n-i-1$, thanks to the first part of the lemma.
\end{proof}

\begin{lemma}\label{cruc2}
The coefficient of degree $n+1$ in $d$ of $\pi_{1,n}^*u_1^n\cdot\pi_{2,n}^*u_2^n\cdots u_n^n$ 
is the same of the one of $(-1)^nc_1(X)^n$, that is $1$.
\end{lemma}

\begin{proof}
An explicit computation yields:
$$
\begin{aligned}
\pi_{1,n}^*u_1^{n}\cdot \pi_{2,n}^*u_2^n\cdots u_n^{n}&\overset{(i)}= \pi_{1,n}^*u_1^{n}\cdot \pi_{2,n}^*u_2^n\cdots \pi_{n}^*u_{n-1}^{n}\bigl(-\pi_n^*c_1^{[n-1]}\cdot u_n^{n-1} \\ &\quad\quad-\cdots-\pi_n^*c_{n-1}^{[n-1]}\cdot u_n-\pi_n^*c_{n}^{[n-1]}\bigr) \\
&\overset{(ii)}=-\pi_{1,n}^*u_1^{n}\cdot \pi_{2,n}^*u_2^n\cdots\pi_{n}^* u_{n-1}^{n}\cdot u_n^{n-1} \cdot \pi_n^*c_1^{[n-1]} \\
&\overset{(iii)}=-\pi_{1,n}^*u_1^{n}\cdot \pi_{2,n}^*u_2^n\cdots \pi_{n}^*u_{n-1}^{n}\cdot u_n^{n-1} \\ &\quad\quad\cdot\pi_n^*\biggl(\pi_{0,n-1}^*c_1(X)+(n-1)\sum_{s=1}^{n-1}\pi_{s,n-1}^*u_s\biggr) \\
&\overset{(iv)}=-\pi_{1,n}^*u_1^{n}\cdot \pi_{2,n}^*u_2^n\cdots \pi_{n}^*u_{n-1}^{n}\cdot u_n^{n-1}\cdot\pi_{0,n}^*c_1(X) \\ &\quad\quad+o(d^{n+1})\\
&=\cdots \\
&\overset{(v)}=(-1)^n\pi_{0,k}^*c_1(X)^n\cdot\pi_{1,k}^*u_1^{n-1}\cdots u_n^{n-1}+o(d^{n+1}) \\
&\overset{(vi)}=(-1)^nc_1(X)^n+o(d^{n+1}).
\end{aligned}
$$
Let us say a few words about the previous equalities. Equality (i) is just relation (\ref{chern3}). Equality (ii) is true for degree reasons: $u_1^{n}\cdot u_2^n\cdots u_{n-1}^{n}\cdot c_l^{[n-1]}$, $l=2,\dots,n$, \lq\lq lives\rq\rq{} on $X_{n-1}$ and has total degree $n(n-1)+l$ which is strictly greater than $n+(n-1)(n-1)=\dim X_{n-1}$, so that $u_1^{n}\cdot u_2^n\cdots u_{n-1}^{n}\cdot c_l^{[n-1]}=0$. Equality (iii) is just relation (\ref{c1}). Equality (iv) follows from the first part of Lemma \ref{cruc1}: $u_1^n\cdots u_s^{n+1}\cdots u_{n-1}^n=o(d^{n+1})$. Equality (v) is obtained by applying repeatedly the second part of Lemma \ref{cruc1}. Finally, equality (vi) is simply integration along the fibers. The lemma is proved.
\end{proof}

Now, look at the coefficient of degree $n+1$ in $d$ of the expression
$$
\mathcal O_{X_n}(\mathbf a)^{n^2}=\bigl(a_1\pi_{1,n}^*u_1+\cdots+a_n u_n\bigr)^{n^2},
$$
where we consider the $a_j$'s as variables: we claim that it is a non identically zero homogeneous polynomial of degree $n^2$. To see this, we just observe that, thanks to Lemma \ref{cruc2}, the coefficient of the monomial $a_1^n\cdots a_n^n$ is $(n^2)!/(n!)^n$.

Hence there exists an $\mathbf a$ which satisfies the hypothesis of Corollary \ref{bigness1} for $k=n$, and Theorem \ref{existence} is proved.

\section{The logarithmic case and proof of Theorem \ref{logexistence}}

Let $D\subset X$ be a simple normal crossing divisor in a compact complex manifold $X$, \emph{i.e.} for each $x\in X$ there exist local holomorphic coordinates $(z_1,\dots,z_n)$ for $X$, centered at $x$, such that locally $D=\{z_1\cdots z_l=0\}$, $0\le l\le n$. Then $T^*_X\langle D\rangle$, the logarithmic cotangent space to $X$ relative to $D$, is well defined and locally free: it is the subsheaf of the sheaf of meromorphic differential forms of the form
$$
\sum_{j=1}^lf_j\,\frac{dz_j}{z_j}+\sum_{k=l+1}^n f_k\,dz_k,
$$
where $f_i\in\mathcal O_{X,x}$ are germs of holomorphic functions in $x$ and the local coordinates are chosen as above. Clearly, we have the following short exact sequence
\begin{equation}\label{log}
0\to T^*_X\to T^*_X\langle D\rangle\to\mathcal O_D\to 0
\end{equation}
and also $T^*_X\langle D\rangle |_{X\setminus D}=T^*_X$.

\subsection{Chern classes computations}

Here, we compute Chern classes for the logarithmic (co)tangent bundle of the pair $(\mathbb P^n,D)$, when $D$ is a smooth projective hypersurface of degree $\deg D=d$. In this case (\ref{log}) become
$$
\xymatrix{
& & & 0 &  \\
0\ar[r] & T^*_{\mathbb P^n}\ar[r] & T^*_{\mathbb P^n}\langle D\rangle\ar[r] & \mathcal O_D\ar[r] \ar[u]& 0, \\
& & & \mathcal O_{\mathbb P^n} \ar[u]& \\
& & & \mathcal O_{\mathbb P^n}(-D) \ar[u] \\
& & & 0 \ar[u] &}
$$
where the vertical arrows are the usual locally free resolution of the structure sheaf of a divisor in $\mathbb P^n$; then
$$
c_{\bullet}(T^*_{\mathbb P^n}\langle D\rangle)=c_{\bullet}(T^*_{\mathbb P^n})c_\bullet(\mathcal O_D)\quad\text{\rm and}\quad 1=c_\bullet(\mathcal O_{\mathbb P^n}(-D))c_\bullet(\mathcal O_D),
$$
so that if $h\in H^2(\mathbb P^n,\mathbb Z)$ is the hyperplane class, we have
$$
1=(1-d\, h)c_\bullet(\mathcal O_D)
$$
and thus $c_\bullet(\mathcal O_D)=1+d\,h+(d\,h)^2+\cdots+(d\,h)^n$. Now, recalling that $c_\bullet(T_{\mathbb P^n})=(1+h)^{n+1}$ and that, for a vector bundle $E$,  $c_j(E^*)=(-1)^jc_j(E)$,
we get the following:

\begin{proposition}
Let $D\subset\mathbb P^n$ be a smooth hypersurface of degree $\deg D=d$. Then the Chern classes of the logarithmic tangent bundle $T_{\mathbb P^n}\langle D\rangle$ are given by
\begin{equation}\label{chernlog}
c_j(T_{\mathbb P^n}\langle D\rangle)=(-1)^jh^j\sum_{k=0}^j(-1)^k\binom{n+1}{k}d^{j-k},
\end{equation}
for $j=1,\dots,n$.
\end{proposition}

\subsection{Logarithmic jet bundles}

Here, we recall the construction due to \cite{D-L01} of logarithmic jet bundles, which is in fact completely analogous to the \lq\lq standard\rq\rq{} one.

So, we start with a triple $(X,D,V)$ where $(X,V)$ is a compact directed manifold and $D\subset X$ is a simple normal crossing divisor whose components $D_{(j)}$ are everywhere transversal to $V$ (that is $T_{D_{(j)}}+V=T_X$ along $D_{(j)}$). Let, as usual, $\mathcal O(V\langle D\rangle)$ be the (locally free in this setting) sheaf of germs of holomorphic vector fields which are tangent to each component of $D$. 

Now, we define a sequence $(X_k,D_k,V_k)$ of logarithmic $k$-jet bundles as in Subsection 2.2: if $(X_0,D_0,V_0)=(X,D,V\langle D\rangle)$, set $X_k=P(V_{k-1})$, $D_k=\pi_{0,k}^{-1}D$ and $V_k$ is the set of logarithmic tangent vectors in $T_{X_k}\langle D_k\rangle$ which project onto the line defined by the tautological line bundle $\mathcal O_{X_{k}}(-1)\subset\pi_k^*V_{k-1}$.

In this case, the direct image formula of Theorem \ref{di} becomes
$$
(\pi_{0,k})_*\mathcal O_{X_k}(m)=\mathcal O(E_{k,m}V^*\langle D\rangle),
$$
where $\mathcal O(E_{k,m}V^*\langle D\rangle)$ is the sheaf generated by all invariant polynomial differential operators in the derivatives of order $1,2,\dots,k$ of the components $f_1,\dots,f_r$ of a germ of holomorphic curve $f\colon (\mathbb C,0)\to X\setminus D$ tangent to $V$, together with the extra functions $\log s_j(f)$ along the $j$-th components $D_{(j)}$ of $D$, where $s_j$ is a local equation for $D_{(j)}$. 

Then, as in the compact case, we have the following:

\begin{theorem}[\cite{D-L01}]
Assume that there exist integers $k,m>0$ and an ample line bundle $A\to X$ such that 
$$
H^0(X_k,\mathcal O_{X_k}(m)\otimes\pi_{0,k}^*A^{-1})\simeq H^0(X,E_{k,m}V^*\langle D\rangle\otimes A^{-1})
$$
has non zero sections $\sigma_1,\dots,\sigma_N$ and let $Z\subset X_k$ be the base locus of these sections. Then every entire holomorphic curve $f\colon\mathbb C\to X\setminus D$ tangent to $V$ is such that $f_{[k]}(\mathbb C)\subset Z$.
\end{theorem}

Just like in the compact case, the locally free sheaf $\mathcal O(E_{k,m}V^*\langle D\rangle)$ arises naturally as a subsheaf of $\mathcal J_{k,m}\langle D\rangle$, of (non necessarily invariant) polynomial differential operators (cf. \cite{D-L01}, \cite{Demailly95}). Moreover, we can endow $\mathcal J_{k,m}V^*\langle D\rangle$ with a natural filtration with respect to the (weighted) degree such that the associated graded bundle is
$$
\operatorname{Gr}^\bullet\mathcal J_{k,m}V^*\langle D\rangle=\bigoplus_{\ell_1+2\ell_2+\cdots+k\ell_k=m}
S^{\ell_1}V^*\langle D\rangle\otimes\cdots\otimes S^{\ell_k}V^*\langle D\rangle.
$$

\subsection{Strategy of the proof and logarithmic case} 

We begin with the following simple observation: the above construction of logarithmic jet bundles is, from the \lq\lq relative\rq\rq{} point of view, exactly the same as in the compact case.
 
This means that the short exact sequences which determine the relations on Chern classes and thus the relative structure of the cohomology algebra are, in the logarithmic case, the same as in Subsection 2.2.

Recall the main points of the proof in the compact case:

\begin{itemize}
\item For $\dim X=n$, go up to the $n$-th projectivized jet bundle, and find a (class of) relatively nef line bundle $\mathcal O_{X_n}(\mathbf a)\to X_n$, with a nontrivial morphism into $\mathcal O_{X_n}(m)$ for some large $m$.
\item Write $\mathcal O_{X_n}(\mathbf a)$ as the difference of two globally nef line bundle, namely 
$$
\bigl(\mathcal O_{X_k}(\mathbf a)\otimes\pi_{0,k}^*\mathcal O_X(2|\mathbf a|)\bigr)\otimes\pi_{0,k}^*\mathcal O_X(-2|\mathbf a|).
$$
\item Compute the \lq\lq Morse\rq\rq{} intersection $F^{n}-nF^{n-1}\cdot G$ for $\mathcal O_{X_n}(\mathbf a)$ and show that, once expressed in term of the degree of $X$, the leading term is the same of $\mathcal O_{X_n}(\mathbf a)^{n^2}$.
\item Use the vanishing Theorem \ref{vanish} to conclude that the term of maximal possible degree in $\mathcal O_{X_k}(\mathbf a)^{n+k(n-1)}$ vanishes for $k<n$.
\item Find a particular non-vanishing monomial in the variables $\mathbf a$, in the expression on maximal possible degree of $\mathcal O_{X_n}(\mathbf a)^{n^2}$.
\end{itemize}

From this discussion, it follows that the only part which remains to be proved in the logarithmic case is an analogous of Theorem \ref{vanish}, all the rest being completely identical: this will be done in the next subsection. 

To conclude the present paragraph, we just observe that the starting point to write $\mathcal O_{X_n}(\mathbf a)$ as the difference of two globally nef line bundles is, for $X$ a smooth projective hypersurface,  that $T^*_X\otimes\mathcal O(2)$ is nef as a quotient of $T^*_{\mathbb P^{n+1}}\otimes\mathcal O(2)$. Thanks to the short exact sequence (\ref{log}), this is the true also in the logarithmic case:
$$
0\to T^*_{\mathbb P^n}\otimes\mathcal O(2)\to T^*_{\mathbb P^n}\langle X\rangle\otimes\mathcal O(2)\to\mathcal O_X(2)\to 0,
$$
and $T^*_{\mathbb P^n}\langle X\rangle\otimes\mathcal O(2)$ in nef as an extension of a nef vector bundle by a nef line bundle (compare with \cite{Div08}).
 
\subsection{Vanishing of global section of low order logarithmic jet differentials} 

We want to prove here the following vanishing theorem.

\begin{theorem}\label{logvanishing}
Let $D\subset\mathbb P^n$ be a smooth irreducible divisor of degree $\deg D=\delta$. Then
$$
H^0(\mathbb P^n,\mathcal J_{k,m}T^*_{\mathbb P^n}\langle D\rangle)=0
$$
for all $m\ge 1$ and $1\le k\le n-1$, provided $\delta\ge 3$.
\end{theorem}

If we pass to the subbundle $E_{k,m}T^*_{\mathbb P^n}\langle D\rangle$ and add some negativity, we get an immediate corollary which is exactly what we need to conclude the proof of Theorem \ref{logexistence}.

\begin{corollary}
Let $D\subset\mathbb P^n$ be a smooth irreducible divisor of degree $\deg D=\delta$ and $A\to\mathbb P^n$ any ample line bundle. Then
$$
H^0(\mathbb P^n,E_{k,m}T^*_{\mathbb P^n}\langle D\rangle\otimes A^{-1})=0
$$
for all $m\ge 1$ and $1\le k\le n-1$, provided $\delta\ge 3$.
\end{corollary}

So, we begin recalling a vanishing theorem contained in \cite{B-R90} for the twisted Schur powers of the cotangent bundle of a smooth complete intersection.

Let $Y=H_1\cap\cdots\cap H_{N-n}\subset\mathbb P^N$ be an $n$-dimensional smooth complete intersection by the hypersurfaces $H_i\subset\mathbb P^N$, with $d_i=\deg H_i$. 

Let $(\lambda)=(\lambda_1,\dots,\lambda_n)$ be a partition of the integer $r=\lambda_1+\cdots+\lambda_n$, with $\lambda_1\ge\cdots\ge\lambda_n\ge 0$, and $T_{(\lambda)}$ the associated Young tableau. Finally, let $t_i$ be the number of cells inside the $i$-th column of $T_{(\lambda)}$ and set $t=\sum_{i=1}^{N-n}t_i$ (take $t_i=0$ if $i>\operatorname{length}(T)$).

Denote with $\Gamma^{(\lambda)}T^*_Y$ the irreducible representation of $\operatorname{Gl}(T^*_Y)$ of highest weight $(\lambda)$ (we refer the reader to \cite{F-H91} for an excellent overview on representations of the general linear group and related things). Then we have the following:

\begin{theorem}[\cite{B-R90}]\label{vanishtw}
If $p<r+\min\{\operatorname{length}(T_{(\lambda)}),d_1-2,\dots,d_{N-n}-2\}$ and $t<n$, then
$$
H^0(Y,\Gamma^{(\lambda)}T^*_Y\otimes\mathcal O_Y(p))=0.
$$
In particular, if $Y\subset\mathbb P^{n+1}$ is a smooth projective hypersurface of degree $\deg Y=d$, then
$$
H^0(Y,\Gamma^{(\lambda)}T^*_Y\otimes\mathcal O_Y(r+p))=0
$$
if $\lambda_n=0$ and $p<\min\{\lambda_1,d-2\}$.
\end{theorem}

From the above theorem, we deduce the following proposition which extends to all dimension a result of El Goul \cite{EG03}.

\begin{proposition}
Let $D\subset\mathbb P^n$ a smooth hypersurface of degree $\deg D\ge 3$. Then
$$
H^0(\mathbb P^n,\Gamma^{(\lambda)}T^*_{\mathbb P^n}\langle D\rangle)=0
$$
for any non increasing $n$-tuple $(\lambda)=(\lambda_1,\dots,\lambda_n)$ with $\lambda_n=0$.
\end{proposition}

\begin{proof}
Consider the standard ramified covering $\mathbb P^{n+1}\supset\widetilde D\to\mathbb P^n$ associated to $D$: if $D$ is given by the homogeneous equation $P(z_0,\dots,z_n)=0$ of degree $\delta$, then $\tilde D\subset\mathbb P^{n+1}$ is cut out by the single equation $z_{n+1}^\delta=P(z_0,\dots,z_n)$.

If we take pullbacks of logarithmic differential forms on $\mathbb P^n$, we obtain an injection $H^0(\mathbb P^n,T^*_{\mathbb P^n}\langle D\rangle)\hookrightarrow H^0(\tilde D, T^*_{\tilde D}\otimes\mathcal O_{\tilde D}(1))$. This is easily seen, as on $\tilde D$ one has
$$
\left.\frac{dP}{P}\right|_{\tilde D}=\left.\frac{d z_{n+1}^\delta}{z_{n+1}^\delta}\right|_{\tilde D}=\left.\delta\,\frac{dz_{n+1}}{z_{n+1}}\right|_{\tilde D},
$$
so that pullbacks of logarithmic forms downstairs give rise to forms with one simple pole along the hyperplane section $\{z_{n+1}=0\}\cap\tilde D$.  

Now, we just have to apply, given the weight $\lambda$, the Schur functors to the injection $H^0(\mathbb P^n,T^*_{\mathbb P^n}\langle D\rangle)\hookrightarrow H^0(\tilde D, T^*_{\tilde D}\otimes\mathcal O_{\tilde D}(1))$, in order to obtain the new injection
$$
H^0(\mathbb P^n,\Gamma^{(\lambda)}T^*_{\mathbb P^n}\langle D\rangle)\hookrightarrow H^0(\tilde D, \Gamma^{(\lambda)}T^*_{\tilde D}\otimes\mathcal O_{\tilde D}(|\lambda|)),
$$
where $|\lambda|=\lambda_1+\cdots+\lambda_n$.
The proposition follows from Theorem \ref{vanishtw} (with $r=|\lambda|$ and $p=0$).
\end{proof}

\subsubsection{End of the proof of the vanishing} 

To conclude the proof of Theorem \ref{logvanishing}, we just need to exclude --- using the same strategy as in \cite{Div08} --- among the irreducible $\operatorname{Gl}(T^*_{\mathbb P^n}\langle D\rangle)$-representa\-tions of the bundle $\mathcal J_{k,m}T^*_{\mathbb P^n}\langle D\rangle$ with $k<n$, Schur powers of the form 
$$
\Gamma^{(\lambda_1,\dots,\lambda_n)}T^*_{\mathbb P^n}\langle D\rangle
$$ 
with $\lambda_n\ne 0$. This is possible thanks to the following elementary lemma.

\begin{lemma}
Let $V$ be a complex vector space of dimension $n$ and $\lambda=(\lambda_1,\dots,\lambda_n)$ such that $\lambda_1\ge\lambda_2\ge\dots\ge\lambda_n\ge 0$. Then
$$
\Gamma^{(\lambda)} V\otimes S^mV\simeq\bigoplus_\mu\Gamma^{(\mu)} V
$$
as $\operatorname{Gl}(V)$-representations, the sum being over all $\mu$ whose Young diagram $T_{(\mu)}$ is obtained by adding $m$ boxes to the Young diagram $T_{(\lambda)}$ of $\lambda$, with no two in the same column.
\end{lemma}

\begin{proof}
This follows immediately by Pieri's formula, see e.g. \cite{F-H91}.
\end{proof}

Note that this implies that among all the irreducible $\operatorname{Gl}(V)$-repre\-sentations of $S^lV\otimes S^m V$, we cannot find terms of type $\Gamma^{(\lambda_1,\dots,\lambda_n)}V$ with $\lambda_i>0$ for $i>2$ (they are all of type $\Gamma^{(l+m-j,j,0,\dots,0)}V$ for $j=0,...,\min\{m,l\}$).

Thus, by induction on the number of factor in the tensor product of symmetric powers,  we easily find:

\begin{corollary}
If $k\le n$, then we have a direct sum decomposition into irreducible $\operatorname{Gl}(V)$-represen\-tations
$$
S^{\ell_1}V\otimes S^{\ell_2}V\otimes\dots\otimes S^{\ell_k}V=\bigoplus_\lambda\nu_\lambda\,\Gamma^{(\lambda)} V,
$$
where $\nu_\lambda\ne0$ only if $\lambda=(\lambda_1,\dots,\lambda_n)$ is such that $\lambda_i=0$ for $i>k$.
\end{corollary}

So, in our hypotheses, the composition series of $\mathcal J_{k,m}T^*_{\mathbb P^n}\langle D\rangle$ has vanishing $H^0$ group, and Theorem \ref{logvanishing} is proved. 

\subsection{Effective Results for the Existence of Logarithmic Jet Differentials in Low Dimension} 

The get the effective results announced in the statements of Theorem \ref{logexistence}, we just compute the algebraic holomorphic Morse inequalities, for $\mathbf a=(2\cdot 3^{n-2},\dots,6,2,1)\in\mathbb N^n$. Hence we get an explicit polynomial in the variable $d$, which has positive leading coefficient, and we compute its largest positive root.

All this is done by implementing a quite simple code on GP/PARI CALCULATOR Version 2.3.2.
The computation complexity blows-up rapidly and, starting from dimension $6$, our computers were not able to achieve any result in a finite time.

\begin{remark}
Although very natural, we don't know if the weight 
$$
\mathbf a=(2\cdot 3^{n-2},\dots,6,2,1)\in\mathbb N^n
$$ 
we utilize is the best possible.
\end{remark}

Here is the code.
\newline
\newline
\texttt{
/*scratch variable*/ 
\newline 
X
\newline \newline 
/*main formal variables*/
\newline 
c=[c1,c2,c3,c4,c5,c6,c7,c8,c9]  /*Chern classes of V<D> on P\^{}n*/
\newline 
u=[u1,u2,u3,u4,u5,u6,u7,u8,u9]  /*Chern classes of OXk(1)*/
\newline 
v=[v1,v2,v3,v4,v5,v6,v7,v8,v9]  /*Chern classes of Vk on Xk*/
\newline 
w=[w1,w2,w3,w4,w5,w6,w7,w8,w9]  /*formal variables*/
\newline 
e=[0,0,0,0,0,0,0,0,0]  /*empty array for logarithmic Chern classes*/
\newline 
q=[0,0,0,0,0,0,0,0,0]  /*empty array for Chern equations*/
\newline 
 \newline /*main*/
\newline Calcul(dim,order)=
\newline \{
\newline local(j,n,N);
\newline n=dim;
\newline r=dim;
\newline k=order;
\newline N=n+k*(r-1);
\newline H(n);
\newline Chern();
\newline B=2*h*3\^{}(k-1);
\newline A=B+u[k];
\newline for(j=1,k-1,A=A+2*3\^{}(k-j-1)*u[j]);
\newline R=Reduc((A-N*B)*A\^{}(N-1));
\newline C=Integ(R);
\newline print("Calculation for order ", k, " jets on logarithmic projective ", n, "-space");
\newline print("Line bundle A= ", A);
\newline print("Line bundle B= ", B);
\newline print("Chern class of A\^{}", N, "-", N, "*A\^{}", N-1, "*B :");
\newline print(C);
\newline E=Eval(C);
\newline print("Evaluation for degree d logarithmic projective ", n,"-space:");
\newline print(E)
\newline 	\}
\newline
\newline
/*compute Chern relations*/
\newline Chern()=  
\newline \{
\newline local(j,s,t);
\newline q[1]=X\^{}r; for(j=1,r,q[1]=q[1]+c[j]*X\^{}(r-j));
\newline for(s=1,r,v[s]=c[s]);
\newline for(s=r+1,9,v[s]=0);
\newline for(t=1,k-1,\
  \newline for(s=1,r,w[s]=v[s]+(binomial(r,s)-binomial(r,s-1))*u[t]\^{}s;
  \newline   for(j=1,s-1,w[s]=w[s]+\
  \newline     (binomial(r-j,s-j)-binomial(r-j,s-j-1))*v[j]*u[t]\^{}(s-j)));
\newline  for(s=1,r,v[s]=w[s]);
  \newline q[t+1]=X\^{}r; for(j=1,r,q[t+1]=q[t+1]+v[j]*X\^{}(r-j)))
\newline \}
\newline 
\newline /*reduction to Chern classes of (P\^{}n,D)*/
\newline
Reduc(p)= 
\newline \{
\newline local(j,a);
\newline a=p;
\newline for(j=0,k-1,\
\newline   a=subst(a,u[k-j],X);
\newline   a=subst(lift(Mod(a,q[k-j])),X,u[k-j]));
\newline a
\newline  \}
\newline 
\newline /*integration along fibers*/
\newline Integ(p)= 
\newline \{
\newline local(j,a);
\newline a=p;
\newline for(j=0,k-1,\
 \newline  a=polcoeff(a,r-1,u[k-j]));
\newline a
\newline \}
\newline
\newline
/*compute Chern classes of degree d logarithmic projective n-space*/
\newline H(n)= 
\newline  \{
\newline local(j,s);
\newline for(s=1,n,\
  \newline e[s]=d\^{}s;
 \newline  for(j=1,s,e[s]=e[s]+(-1)\^{}j*(d)\^{}(s-j)*binomial(n+1,j));
 \newline  e[s]=(-1)\^{}s*e[s])
\newline  \}
\newline
\newline
 /*evaluation in terms of the degree*/
\newline Eval(p)=
\newline \{
\newline local(a,s);
\newline a=p;
\newline for(s=1,r,a=subst(a,c[s],e[s]));
\newline subst(a,h,1)*d
\newline \}
}

\end{document}